\documentclass{gtpart}
\usepackage{enumerate}
\usepackage[pdftex]{color}% Obvia
\usepackage{amsmath,amsfonts,amssymb}
\usepackage[all]{xy}
\usepackage{graphicx}
\usepackage{fancyhdr}
\usepackage{hyperref}
\hypersetup{hyperref,backref,urlcolor=black,citecolor=black, linkcolor=black, colorlinks=true}
\title{Representation stability\\ for the cohomology of the moduli space $\mathcal{M}_{g}^n$}

\author{Rita Jimenez Rolland}
\givenname{}
\address{Department of Mathematics,
University of Chicago,
Chicago, IL 60637}
\email{atir83@math.uchicago.edu}
\urladdr{http://www.math.uchicago.edu/~atir83/}


\arxivreference{}
\arxivpassword{}

\volumenumber{}
\issuenumber{}
\publicationyear{}
\papernumber{}
\startpage{}
\endpage{}
\doi{}
\MR{}
\Zbl{}
\received{}
\revised{}
\accepted{}
\published{}
\publishedonline{}
\proposed{}
\seconded{}
\corresponding{}
\editor{}
\version{}

\newtheorem{definition}{Definition}

\newtheorem{theo}{Theorem}[section]

\newtheorem*{rtheo1}{Theorem 1.1}
\newtheorem*{rtheo2}{Theorem 1.7}
\newtheorem*{rcor4}{Corollary 1.3}

\newtheorem{cor}[theo]{Corollary}
\newtheorem{lemma}[theo]{Lemma}
\newtheorem{prop}[theo]{Proposition}

\DeclareMathOperator{\Mod}{Mod}
\DeclareMathOperator{\PMod}{PMod}
\DeclareMathOperator{\Push}{Push}
\DeclareMathOperator{\Diff}{Diff}
\DeclareMathOperator{\PDiff}{PDiff}

\DeclareMathOperator{\Aut}{Aut}

\newcommand{\BPMod}{B\PMod}
\newcommand{\BMod}{B\Mod}
\newcommand{\BDiff}{B\Diff}
\newcommand{\BPDiff}{B\PDiff}

\begin{document}
\begin{abstract}
Let $\mathcal{M}_{g}^n$ be the moduli space of  Riemann surfaces of genus $g$ with $n$ labeled marked points. %In general, the sequence $\{\mathcal{M}_{g}^n\}_{n=1}^{\infty}$ does not satisfy (co)homological stability, even rationally. It is known for example that dim $H^2(\mathcal{M}_{g}^n;\mathbb{Q})=n+1$ for $g\geq 4$.  
We prove that, for $g\geq 2$, the cohomology groups $\{H^i(\mathcal{M}_{g}^n;\mathbb{Q})\}_{n=1}^{\infty}$ form a sequence of $S_n$-representations which is representation stable in the sense of Church-Farb \cite{CHURCH_FARB}. In particular this result applied to the trivial $S_n$-representation implies rational ``puncture homological stability'' for the mapping class group $\Mod_g^n$. % This has been recently proved  for $\Mod_{g,r}^n$ (with $r\geq1$)  by Hatcher-Wahl  \cite{HATCHER_WAHL} with integral coefficients.
We obtain representation stability for sequences $\{H^i(\PMod^n(M);\mathbb{Q})\}_{n=1}^{\infty}$, where $\PMod^n(M)$ is the mapping class group of many connected manifolds $M$ of dimension  $d\geq 3$ with centerless fundamental group; and for sequences $\{H^i\big(\BPDiff^n(M);\mathbb{Q}\big)\}_{n=1}^{\infty}$, where $\BPDiff^n(M)$ is the classifying space of the subgroup $\PDiff^n(M)$ of diffeomorphisms of $M$ that fix pointwise $n$ distinguished points in $M$. 
\end{abstract}
\maketitle
\section{Introduction}

\textbf{Notation:} 
Let $\Sigma_{g,r}$ be a compact orientable surface of genus $g\geq 0$ with $r\geq 0$ boundary components and let $p_1,\ldots, p_n$ be distinct points in the interior of $\Sigma_{g,r}$. The {\it mapping class group}  $\Mod_{g,r}^n$ is the group of isotopy classes of orientation-preserving self-diffeomorphisms of $\Sigma_{g,r}^n:= \Sigma_{g,r}-\{p_1,\ldots, p_n\}$ that restrict to the identity on the boundary components. The {\it pure mapping class group} $\PMod_{g,r}^n$ is defined analogously by asking that the punctures remain fixed pointwise. If $r=0$ or $n=0$, we omit it from the notation.%\bigskip
 
The homology groups of the pure mapping class group $\PMod_{g}^n$ are of interest (among other reasons) due to their relation with the topology of the moduli space $\mathcal{M}_{g}^n$  of genus $g$ Riemann surfaces with $n$ labeled marked points (i.e. $n$-pointed  non-singular projective curves of genus $g$). The space $\mathcal{M}_{g}^n$  is a rational model for the classifying space $\BPMod_{g}^n$ for $g\geq 2$. Hence
\begin{equation}\label{moduli}
H^*(\mathcal{M}_{g}^n;\mathbb{Q})\approx H^*(\PMod_{g}^n;\mathbb{Q}).
\end{equation}
 We refer the  reader to \cite{FARBMARG}, \cite{HAINLOO}, \cite{KIRWAN} and \cite{HARERmoduli} for more about the relation between $\mathcal{M}_{g}^n$ and $\PMod_g^n$.

One basic question is to understand how, for a fixed $i\geq 0$, the cohomology groups $H^i(\PMod_{g,r}^n;\mathbb{Q})$ change as we vary the parameters $g$, $r$ and $n$, in particular when the parameters are very large with respect to $i$. 
It is a classical result by Harer \cite{HARER} that the group $H^i(\PMod_{g,r}^n;\mathbb{Z})$ depends only on $n$ provided that $g$ is large enough. The major goal of this paper is to understand how the cohomology $H^i(\PMod_{g,r}^n;\mathbb{Q})$ changes  as we vary the number of punctures $n$.

\subsection*{Genus and puncture homological stability}\label{Motiv}
It is known that the groups $\PMod_{g,r}^n$ and  $\Mod_{g,r}^n$ satisfy ``genus homological stability'': 

\begin{center}
\textit{For fixed $i,n\geq 0$ the homology groups  $H_i(\PMod_{g,r}^n;\mathbb{Z})$ and $H_i(\Mod_{g,r}^n;\mathbb{Z})$ }
\textit{do not depend on the parameters $g$ and $r$, for $g\gg i$.}
\end{center}
This was first proved in the $1980$'s  by Harer \cite{HARER}  and the stable ranges have been improved since then by the work of several people (see Wahl's survey \cite{WAHL}).

An additional stabilization map can be defined by increasing the number of punctures. In the case of surfaces with non-empty boundary, we can consider a map $\Sigma_{g,r}^n\rightarrow \Sigma_{g,r}^{n+1}$ by gluing a punctured cylinder to one of the boundary components of $\Sigma_{g,r}^n$. This map gives a homomorphism $$\mu_n\colon \Mod_{g,r}^n\rightarrow \Mod_{g,r}^{n+1}.$$ In \cite[Proposition 1.5]{HATCHER_WAHL}, Hatcher and Wahl proved that the map $\mu_n$ induces an isomorphism in $H_i(-;\mathbb{Z})$ if $n\geq 2i+1$ (for fixed $g\geq 0$ and $r>0$). Puncture stability for closed surfaces follows, as it is known that 
\begin{equation*}
H_i(\Mod_{g,1}^n;\mathbb{Z})\approx H_i(\Mod_{g}^n;\mathbb{Z})\text{ for }g \geq \frac{3}{2}i
 \end{equation*}
(see \cite[Theorem 1.2]{WAHL}). Handbury proved  this ``puncture homological stability'' for non-orientable surfaces in \cite{HANDBURY} with techniques that can also be applied to the orientable case. When the surface is a punctured disk this is Arnold's classical stability theorem for the cohomology of braid groups $B_n$ \cite{ARNOLD}.
Together, puncture and genus stability imply that the homology of the mapping class group of an orientable surface stabilizes with respect to connected sum with any surface. 

On the other hand, for the pure mapping class groups, attaching a punctured cylinder to $\Sigma_{g,r}^n$ also induces homomorphisms $$\mu_n\colon \PMod_{g,r}^{n}\rightarrow \PMod_{g,r}^{n+1},$$ when $r>0$. Hence we can ask whether $\PMod_{g,r}^n$ satisfies or not puncture homological stability. 

The homology groups of $\PMod_{g,r}^n$ are largely unknown, apart from some low dimensional cases such as: 
\begin{equation*}
H_1(\PMod_{g,r}^n;\mathbb{Z})=0\text{ for } g\geq 3
\end{equation*}
(see \cite[Theorem 5.2]{FARBMARG} for a proof). Furthermore, 
\begin{equation*}
H_2(\PMod_{g,r}^n;\mathbb{Z})\approx H_2(\Mod_{g,r+n};\mathbb{Z})\oplus \mathbb{Z}^{n}\text{ for }g\geq 3
\end{equation*}
(this is \cite[Corollary 4.5]{KORKMAZ}, but the original computation for $g\geq 5$ is due to Harer \cite{HARERSecond}). 

Even if the case of the first homology group is not representative, we notice that the rank of $H_2(\PMod_{g,r}^n;\mathbb{Z})$ blows up as $n\rightarrow +\infty$. Moreover,  the pure braid groups $P_n\approx\PMod_{0,1}^n$ fail in each dimension $i\geq 1$ to satisfy homological stability \cite[Section 4]{CHURCH_FARB}.  This suggests to us the failure of puncture homological stability in the general case.

For large $g$,  B{\"o}digheimer and Tillmann  results in \cite{TILLMANN}, combined with Madsen-Weiss, give explicit calculations, although we do not discuss them in this paper.

\subsection*{Main result}
We want to compare $H^i(\PMod_{g,r}^{n};\mathbb{Q})$ as the number of punctures $n$ varies. The natural inclusion $\Sigma_{g,r}^{n+1}\hookrightarrow \Sigma_{g,r}^{n}$ induces the \textit{forgetful map} $$f_n\colon\PMod_{g,r}^{n+1}\rightarrow \PMod_{g,r}^{n}.$$
Notice that $f_n$ is a left inverse for the map $\mu_n$ above, when $r>0$, but can be defined even for surfaces without boundary. 
This map allows us to relate the corresponding cohomology groups:
$$f_n^*\colon H^*(\PMod_{g,r}^{n};\mathbb{Q})\rightarrow H^*(\PMod_{g,r}^{n+1};\mathbb{Q}).$$
Observe that $f_n^*$ is also induced by the \textit{forgetful morphism} between moduli spaces $\mathcal{M}_{g}^{n+1}\rightarrow \mathcal{M}_{g}^{n}$.

The key idea is to consider the natural action of the symmetric group $S_n$ on $\mathcal{M}_{g}^n$ given by permuting the $n$ labeled marked points. Thus we can regard $H^i(\mathcal{M}_{g}^n;\mathbb{Q})$ as rational $S_n$-representations and compare them through the maps $f_n^i$. Moreover, we notice that the map $f_n^i$ is equivariant with respect to the standard inclusion $S_n\hookrightarrow S_{n+1}$. In Section \ref{H_2} below we explicitly compute the $S_n$-representation $H^2(\PMod_{g}^n;\mathbb{Q})$ and its decomposition into irreducibles.

Roughly speaking, we say that a sequence of $S_n$-representations $\{V_n\}$ with linear maps $\phi_n\colon V_n\rightarrow V_{n+1}$ equivariant with respect to $S_n\hookrightarrow S_{n+1}$  is \textit{representation stable} if for sufficiently large $n$ the following conditions hold: the maps $\phi_n$ are injective; the image  $\phi_n(V_n)$ generates  $V_{n+1}$  as an $S_{n+1}$-module, and the decomposition of $V_n$ into irreducibles can be described independently of $n$.  This notion was introduced by Church-Farb in \cite{CHURCH_FARB}. The precise definition of representation stability is stated in Section \ref{Subsec1} below. 

 Hence, instead of asking if  $f_n^i$ is an isomorphism or not (puncture cohomological stability), we consider the question of whether the cohomology groups of the pure mapping class group satisfy representation stability. In \cite[Theorem 4.2]{CHURCH_FARB} Church-Farb prove that the sequence $\{H^i(P_n;\mathbb{Q}), f_n^i\}_{n=1}^{\infty}$ is representation stable.  Our main result shows that this is also the case for the pure mapping class group.

\begin{theo}\label{MAIN}
 For any $i\geq 0$ and $g\geq 2$ the sequence of cohomology groups $\{H^i(\PMod_g^n;\mathbb{Q})\}_{n=1}^{\infty}$ is monotone and uniformly representation stable with stable range $$n\geq \text{min }\{ 4i+2(4g-6)(4g-5), 2i^2+6i\}.$$
\end{theo}

 Our arguments work for hyperbolic non-closed surfaces (Theorem \ref{BOUNDARY}). Hence Harer's homological stability and our main theorem imply that, as an $S_n$-representation, $H^i(\PMod_{g,r}^n;\mathbb{Q})$ is independent of $g$, $r$ and $n$, provided $n$ and $g$ are large enough.

By (\ref{moduli}), Theorem \ref{MAIN} can be restated as follows.

\begin{cor}[Representation stability for the cohomology of the moduli space $\mathcal{M}_{g}^{n}$]\label{MAINModuli}
 For any $i\geq 0$ and $g\geq 2$ the sequence of cohomology groups $\{H^i(\mathcal{M}_{g}^{n};\mathbb{Q})\}_{n=1}^{\infty}$ is monotone and uniformly representation stable with stable range $$n\geq \text{min }\{ 4i+2(4g-6)(4g-5), 2i^2+6i\}.$$
\end{cor}

\textbf{Remark:} In \cite[Theorem 1.1]{TILLMANN} B{\"o}digheimer and Tillmann proved that $$B(\PMod ^n_{\infty,r})^+\simeq \BMod_{\infty}^+\times(\Bbb CP^\infty)^{n}.$$
\noindent Together with Harer's homological stability theorem this implies that, in dimensions $*\leq g/2$, 
 $$H^*(\PMod ^n_{g,r};\mathbb{Q}) \approx H^* (\PMod_{g,r};\mathbb{Q}) \otimes \big(H^*(\Bbb CP^\infty;\mathbb{Q})\big)^{\otimes n}$$ $$ \hspace{20mm}\approx H^* (\PMod_{g,r};\mathbb{Q}) \otimes \mathbb{Q}[x_1,\ldots,x_n],$$
\noindent  where each $x_i$ has degree $2$. The action of the symmetric group $S_n$  on the left hand side corresponds to permuting the $n$ factors $\Bbb CP^\infty$. In other words, it is given by the action of $S_n$ on the polynomial ring in $n$ variables by permutation of the variables $x_i$. On the other hand, Church and Farb proved in \cite[Section 7]{CHURCH_FARB} that representation stability holds for the $S_n$-action on the polynomial ring in $n$ variables. Hence B{\"o}digheimer and Tillmann result implies that for $i\leq g/2$ representation stability holds for 
 $\{H^i(\PMod_{g,r}^n;\mathbb{Q})\}_{n=1}^{\infty}$. Notice that this only holds for large $g$ with respect to $i$. In contrast, our Theorem \ref{MAIN} and Theorem \ref{BOUNDARY} give uniform representation stability and monotonicy for arbitrary  $g\geq 0$ such that  $2g+r+s>2$ and large $n$. 

\subsection*{Puncture (co)homological stability for $\Mod_{g}^n$}
Our main result, Theorem \ref{MAIN}, implies cohomological stability for $\Mod_{g}^n$ with twisted rational coefficients (see Section \ref{Conseq}). For any  partition $\lambda$, we denote the corresponding irreducible $S_n$-representation by $V(\lambda)_n$, as we explain in Section \ref{Subsec1} below. A transfer argument gives the proof of the following corollary of Theorem \ref{MAIN}.

\begin{cor}\label{TWISTSTAB}
 For any  partition $\lambda$, the sequence $\{H^i\big(\Mod_g^n;V(\lambda)_n\big)\}_{n=1}^{\infty}$ of twisted cohomology groups satisfies classical cohomological stability: for fixed $i\geq 0$ and $g\geq 2$, there is an isomorphism  $$H^i\big(\Mod_g^n;V(\lambda)_n\big)\approx H^i\big(\Mod_g^{n+1};V(\lambda)_{n+1}\big),$$ if $n\geq\text{min }\{ 4i+2(4g-6)(4g-5), 2i^2+6i\}$.
\end{cor}

In \cite[Proposition 1.5]{HATCHER_WAHL}, Hatcher-Wahl obtained integral puncture homological stability for for the mapping class group of surfaces with non-empty boundary and established a stable range linear in $i$. Plugging in the trivial representation $V(0)_n$ into Corollary \ref{TWISTSTAB}, we recover rational puncture homological stability for $\Mod_{g}^n$. The stable range that we obtain either depends on the genus $g$ of the surface or is quadratic in $i$ (see Corollary \ref{CLASSTAB}).  Nonetheless, our approach by representation stability is completely different from the classical techniques used in the proofs of homological stability. Furthermore, we believe that our proof gives yet another example of how the notion of representation stability can give meaningful answers where classical stability fails.

\subsection*{Pure mapping class groups for higher dimensional manifolds }
\textbf{Notation: }Let $M$ be a connected, smooth manifold and let $p_1,\ldots, p_n$ be distinct points in the interior of $M$. We define the \textit{mapping class group} to be the group $$\Mod^n(M):=\pi_0\big(\Diff^n (M)\big)$$
where $\Diff^n(M)$ is the subgroup of diffeomorphisms in $\Diff (M \text{ rel } \partial M)$ that leave invariant the set of  points $\{p_1,\ldots,p_n\}$. 
Similarly, we let $\PDiff^n(M)$ be the subgroup of diffeomorphims in $\Diff (M \text{ rel } \partial M)$ that fix the points $p_1,\ldots,p_n$ pointwise and  the \textit{pure mapping class group } is the group $$\PMod^n(M):= \pi_0(\PDiff^n(M)).$$

In Section \ref{Subsec5} we give a proof of representation stability for the sequence  $\{H^i(G^n;\mathbb{Q})\}_{n=1}^{\infty}$ for any group $G$. This is Proposition \ref{PRODUCT} below. We show how to use this result and the ideas developed in this paper to establish the analogue of Theorem \ref{MAIN} and Corollary \ref{TWISTSTAB} for the pure mapping class groups of some connected manifolds of higher dimension.

\begin{theo}\label{MAINDIM>3}
Let $M$ be a smooth connected manifold of dimension $d\geq 3$ such that $\pi_1(M)$ is of type $FP_{\infty}$ (e.g. $M$ compact). Suppose that $\pi_1(M)$ has trivial center or that $\Diff(M)$ is simply connected. If $\Mod(M)$ is a group of type $FP_{\infty}$, then for any $i\geq 0$ the sequence of cohomology groups $\{H^i\big(\PMod^n(M);\mathbb{Q}\big)\}_{n=1}^{\infty}$ is monotone and uniformly representation stable with stable range $n\geq 2i^2+4i$.
\end{theo}

\begin{cor}\label{TWISTSTABDIM>3}
 Let $M$ be as in Theorem \ref{MAINDIM>3}. For any  partition $\lambda$, the sequence of twisted cohomology groups $\{H^i\big(\Mod^n(M);V(\lambda)_n\big)\}_{n=1}^{\infty}$  satisfies classical homological stability: for fixed $i\geq 0$, there is an isomorphism  $$H^i\big(\Mod^n(M);V(\lambda)_n\big)\approx H^i\big(\Mod^{n+1}(M);V(\lambda)_{n+1}\big)\text{ if } n\geq 2i^2+4i.$$
\end{cor}

Hatcher-Wahl proved integral puncture homological stability for mapping class group of connected manifolds with boundary of dimension $d\geq 2$ in  \cite[Proposition 1.5]{HATCHER_WAHL}. Our Corollary \ref{TWISTSTABDIM>3}, applied to the trivial representation, gives rational puncture homological stability for $\Mod^n(M)$ for manifolds $M$ that satisfy the hypothesis of Theorem \ref{MAINDIM>3}, even if the manifold has empty boundary.

\subsection*{Classifying spaces for diffeomorphism groups}
Ezra Getzler and Oscar Randal-Williams pointed out to me that the same ideas also give representation stability for the rational cohomology groups of the classifying space $\BPDiff^n(M)$ of the group $\PDiff^n(M)$ defined above.

\begin{theo}\label{MAINBDIFF>3}
Let $M$ be a smooth, compact and connected manifold of dimension $d\geq 3$ such that $\BDiff(M\text{ rel }\partial M)$ has the homotopy type of CW-complex with finitely many cells in each dimension. Then for any $i\geq 0$ the sequence of cohomology groups $\{H^i\big(\BPDiff^n(M);\mathbb{Q}\big)\}_{n=1}^{\infty}$ is monotone and uniformly representation stable with stable range $n\geq 2i^2+4i$.
\end{theo}

\noindent The details are described at the end of the paper in Section \ref{LAST}.

\subsection*{Outline of the proof of Theorem \ref{MAIN}}

 The proof of Theorem \ref{MAIN} is presented in Section \ref{SectionMain} and relies on the existence of the Birman exact sequence which realizes $\pi_1(C_n(\Sigma_{g}))$ as a subgroup of $\PMod_g^n$. Here $C_n(\Sigma_{g,r})$ denotes the configuration space of $n$ distinct ordered points in the interior of $\Sigma_{g,r}$. Then for each $n$ we can consider the associated Hochschild-Serre spectral sequence $E_*(n)$, which allows us to relate  $H^*(\PMod_{g}^n;\mathbb{Q})$ with $H^*\big(\pi_1(C_n(\Sigma_{g}));\mathbb{Q}\big)$. Following ideas of Church in \cite{CHURCH}, we use an inductive argument to show that the terms in each page of the spectral sequence are uniformly representation stable and thus we conclude the result in Theorem \ref{MAIN} from the $E_{\infty}$-page. 

The notion of \textit{monotonicity} for a sequence of $S_n$-representations introduced in \cite{CHURCH} is key in our inductive argument on the pages of the spectral sequence. The base of the induction is monotonicity and representation stability for the terms in the $E_2$-page of the Hochschild-Serre spectral sequence. In order to prove this, we introduce, in Section \ref{Coeff} below, the notion of a consistent sequence of rational $S_n$-representations \textit{compatible with $G$-actions} and prove the following general result which we hope will be useful in future computations.

\begin{theo}[Representation stability with changing coefficients]\label{GSTAB} Let $G$ be a group of type $FP_{\infty}$. Consider a consistent sequence $\{V_n,\phi_n\}_{n=1}^{\infty}$ of finite dimensional rational representations of  $S_n$ compatible with $G$-actions. If the sequence $\{V_n,\phi_n\}_{n=1}^{\infty}$ is monotone and uniformly representation stable with stable range $n\geq N$, then for any integer $p\geq 0$, the sequence $\{H^p(G;V_n), \phi_n^*\}_{n=1}^{\infty}$ is monotone and uniformly representation stable with the same stable range.
\end{theo}

Monotonicity and uniform representation stability for the $E_2$-page follow from Theorem \ref{GSTAB}, as a consequence of the following result by Church \cite[Theorem 1]{CHURCH}.%, in the particular case where $M$ is a closed surface.

\begin{theo}[Church]\label{CHURCH_RS}
For any connected orientable manifold $M$ of finite type and any $q\geq 0$, the cohomology groups $\{H^q(C_n(M);\mathbb{Q})\}$ of the ordered configuration space $C_n(M)$ are monotone and uniformly representation stable, with stable range $n\geq 2q$ if $\dim$  $M\geq 3$ and stable range $n\geq 4q$ if $\dim$  $M=2$.
\end{theo}

\begin{large}  {\bf Acknowledgments. }\end{large}I want to thank Benson Farb for being a patient guide and for his unconditional support along the way. I am indebted to him and Tom Church for all the useful conversations and for developing the language and the ideas which made possible the realization of this work.  I am grateful to Allen Hatcher and Peter May for their generous answers to my questions. I thank Ezra Getzler and Oscar Randal-Williams for suggesting additional applications, and Ulrike Tillmann for useful comments. I am thankful to the anonymous referee for pointing out a relevant reference that I was missing.
 
\section{Preliminaries}
The precise definition of representation stability and monotonicity are stated below. We also recall some useful facts about group extensions and cohomology of groups.
\subsection{Representation Stability and Monotonicity}\label{Subsec1}
Recall that a {\it rational $S_n$-representation} is a $\mathbb{Q}$-vector space equipped with a linear $S_n$-action. The irreducible representations of $S_n$ are classified by partitions $\lambda=(\lambda_1\geq\cdots\geq\lambda_l)$ of $n$ ( with $\lambda_1+\cdots+\lambda_l=n$). We denote the corresponding irreducible $S_n$-representation by $V_{\lambda}$. Every  $V_{\lambda}$ is defined over $\mathbb{Q}$ and any $S_n$-representation decomposes over $\mathbb{Q}$ into a direct sum of irreducibles (\cite{FULTON_HARRIS} is a standard reference).

If $\lambda$ is any  partition of $k$, then for any $n\geq k+ \lambda_1$ the \textit{padded partition} $\lambda[n]$ of $n$ is given by $\lambda[n]=(n-k,\lambda_1,\cdots,\lambda_l)$. Keeping the notation from \cite{CHURCH_FARB}  we set $V(\lambda)_n=V_{\lambda[n]}$  for any $n\geq k+\lambda_1$. Every irreducible $S_n$-representation is of the form $V(\lambda)_n$ for a unique partition $\lambda$.

The notion of representation stability for different families of groups was first defined in \cite{CHURCH_FARB}. We recall this notion for the case of $S_n$-representations.

\begin{definition}
 A sequence $\{V_n\}_{n=1}^{\infty}$ of finite dimensional rational $S_n$-representations with linear maps $\phi_n\colon V_n\rightarrow V_{n+1}$ is said to be \textit{uniformly representation stable with stable range $n\geq N$} if the following conditions are satisfied for all $n\geq N$:
\begin{itemize}
 \item[0.]{\textbf{Consistent Sequence.} The maps $\phi_n\colon V_n\rightarrow V_{n+1}$  are equivariant with respect to the natural inclusion $S_n\hookrightarrow S_{n+1}$.}
 \item[I.]{\textbf{Injectivity.} The maps $\phi_n\colon V_n\rightarrow V_{n+1}$ are injective.}
 \item[II.]{\textbf{Surjectivity.} The $S_{n+1}$-span of  $\phi_n(V_n)$ equals $V_{n+1}$.} 
 \item[III.]{\textbf{Uniformly  Multiplicity Stable with range $n\geq N$.} For each partition $\lambda$, the multiplicities $c_{\lambda}(V_n)$  of $V(\lambda)_n$ in $V_n$ are constant for all $n\geq N$.} 
\end{itemize}
\end{definition}

The notion of monotonicity introduced in \cite{CHURCH} will be key in our argument.

\begin{definition}
A consistent sequence $\{V_n\}_{n=1}^{\infty}$ of $S_n$-representations with injective maps $\phi_n\colon V_n\hookrightarrow V_{n+1}$ is \textit{monotone} for $n\geq N$ if for each subspace $W<V_n$ isomorphic to $V(\lambda)_n^{\oplus k}$, the $S_{n+1}$-span of $\phi_n(W)$ contains $V(\lambda)_{n+1}^{\oplus k}$ as a subrepresentation for $n\geq N$.
\end{definition}

Now we point out the properties of monotone sequences that are useful for our purpose. These results are proven in \cite[Sections 2.1 and 2.2]{CHURCH}.

\begin{prop}\label{SUB_QUO}
Given $\{W_n\}<\{V_n\}$, if the sequence $\{V_n\}$ is monotone then so is $\{W_n\}$. If $\{V_n\}$ and  $\{W_n\}$ are monotone and uniformly representation stable with stable range $n\geq N$, then $\{V_n/W_n\}$ is monotone and representation stable for $n\geq N$. Conversely, if $\{W_n\}$ and  $\{V_n/ W_n\}$ are monotone and uniformly representation stable with stable range $n\geq N$, then $\{V_n\}$ is monotone and uniformly representation stable for $n\geq N$.
\end{prop}

\begin{prop}\label {KER_IM}
Let $\{V_n\}$ and $\{W_n\}$ be monotone sequences for $n\geq N$, and assume that $\{V_n\}$  is uniformly representation stable for $n\geq N$. Then for any consistent sequence of maps $f_n\colon V_n\rightarrow W_n$ that makes the following diagram commutative
\begin{equation*}
\xymatrix{
V_n\ar[r]^{f_n}\ar[d]_{\phi_n}& W_n\ar[d]_{\psi_n}\\
V_{n+1}\ar[r]^{f_{n+1}}& W_{n+1},}
\end{equation*}

 the sequences $\{\text {ker}f_n\}$ and $\{\text{im}f_n\}$ are monotone and uniformly representation stable for $n\geq N$.
\end{prop}\medskip
The previous propositions apply also to $V(\lambda)_n$ for a single partition $\lambda$. In particular to the case of the trivial representation $V(0)_n$.
\begin{prop}\label{PARTITION}
For a fixed partition $\lambda$, assuming monotonicity just for $V(\lambda)_n^{\otimes k}$, Propositions \ref{SUB_QUO} and \ref{KER_IM} hold if we replace ``uniform representation stability'' by ``the multiplicity of $V(\lambda)_n$ is stable''.
\end{prop}

\subsection{On the cohomology of group extensions} \label{Subsec2}
A \textit{group extension} of a group $Q$ by a group $H$ is a short exact sequence of groups 
\begin{equation}\label{EXT}
 1\rightarrow H\rightarrow G\rightarrow Q\rightarrow 1.
\end{equation}
Given a $G$-module $M$, the conjugation action  $(h,m)\mapsto(ghg^{-1},g\cdot m)$ of $G$ on $(H,M)$ 
induces an action of $G/H\cong Q$ on $H^*(H;M)$ as follows. Let $F\rightarrow\mathbb{Z}$ be a projective resolution of $\mathbb{Z}$ over $\mathbb{Z}G$ and consider the diagonal action of $G$ in the cochain complex $\mathcal{H}om(F,M)$ given by $f\mapsto [x\mapsto g\cdot f(g^{-1}\cdot x)]$, for $f\in\mathcal{H}om(F,M)$ and $g\in G$. This action restricts to the subcomplex $\mathcal{H}om_H(F,M)$ where $H$ acts trivially by definition, hence we get an induced action of $Q\cong G/H$ on $\mathcal{H}om_H(F,M)$. But the cohomology of this complex is $H^*(H;M)$, giving the desired action of $Q$ on $H^*(H;M)$.

The {\it cohomology Hochschild-Serre spectral sequence} for the group extension (\ref{EXT}) is a first quadrant spectral sequence converging to $H^*(G;M)$  whose $E_2$ page is of the form 
$$E_2^{p,q}=H^p(Q;H^q(H;M)).$$
Furthermore, from the construction of the Hochschild-Serre spectral sequence it can be shown that this spectral sequence is natural in the following sense. Assume we have group extensions (I) and (II) and  group homomorphisms $f_H$ and $f_G$ making the following diagram commute
\begin{equation*}
\xymatrix{
1\ar[r] &H_1 \ar[r]\ar[d]_{f_H} &G_1\ar[r]\ar[d]_{f_G} &Q\ar[r]\ar@{=}[d]_{\text{id}} &1&\text{ (I)}\\
1\ar[r] &H_2 \ar[r] &G_2\ar[r] &Q\ar[r] &1&\text{ (II)}}
\end{equation*}

Then the induced map $$f_H^*\colon H^*(H_2;\mathbb{Q})\rightarrow H^*(H_1;\mathbb{Q})$$ is $Q$-equivariant. Moreover, if $'E_*$ and $''E_*$ denote the Hochschild-Serre spectral sequences corresponding to the extensions (I) and (II), we have
\begin{itemize}
 \item [1)] {Induced maps $(f_H)^*_r\colon {''E_r^{p,q}}\rightarrow {'E_r^{p,q}}$ that commute with the differentials.} 
\item [2)] {The map $(f_G)^*\colon H^*(G_2;\mathbb{Q})\rightarrow H^*(G_1;\mathbb{Q})$ preserves the  natural filtrations of $H^*(G_1;\mathbb{Q})$ and $H^*(G_2;\mathbb{Q})$ inducing a map on the succesive quotients of the filtrations which is the map
$$(f_H)^*_{\infty}\colon {''E_{\infty}^{p,q}}\rightarrow {'E_{\infty}^{p,q}}.$$}
\item [3)] {The map $(f_H)^*_2\colon {''E_2^{p,q}}\rightarrow {'E_2^{p,q}}$ is the one induced by the group homomorphisms id$\colon Q\rightarrow Q$ and $f_H\colon H_1\rightarrow H_2$.}
\end{itemize}

For an explicit description of the Hochschild-Serre spectral sequence we refer the reader to \cite{BROWN} and \cite{MACLANE} (where it is called the Lyndon spectral sequence).

\section{The second cohomology $H^2(\mathcal{M}_{g}^n;\mathbb{Q})$}\label{H_2}

In this section we understand the consistent sequence of $S_n$-representations $$\{H^2(\PMod_{g}^n;\mathbb{Q}),\text{ }f_n^2\}$$ to give an explicit discussion of the phenomenon of representation stability proved on Theorem \ref{MAIN}.

The second cohomology group is given by:
\begin{equation}\label{secondcoho}
H^2(\mathcal{M}_{g,n};\mathbb{Q})\approx H^2(\PMod_{g}^n;\mathbb{Q})\approx H^2(\Mod_{g,n};\mathbb{Q})\oplus \mathbb{Q}^{n},\text{ for } g\geq 3.
\end{equation}
We want to compare $H^2(\PMod_g^n;\mathbb{Q})$ through the forgetful maps $$f_n^2\colon H^2(\PMod_{g}^n;\mathbb{Q})\rightarrow H^2(\PMod_{g}^{n+1};\mathbb{Q}).$$ We already know that  $f_n^2$ is never an isomorphism  (failure of homological stability). Instead, we consider $H^2(\PMod_g^n;\mathbb{Q})$ as an $S_n$-representation and we investigate how those representations depend on the parameter $n$. 
When $g\geq 4$,   $H^2(\Mod_{g,n};\mathbb{Q})\approx \mathbb{Q}$ \cite{HARERSecond} and the $S_n$-action on this summand is trivial. On the other hand, the summand $\mathbb{Q}^{n}$ is generated by classes $\tau_i\in H^2(\PMod_g^n;\mathbb{Q})$ ($i=1,\ldots, n$) corresponding to the central extensions $\PMod(X_i)$:
$$1\rightarrow \mathbb{Z}\rightarrow \PMod(X_i)\rightarrow \PMod_g^n\rightarrow 1.$$
The right map above is induced from the inclusion $X_i:=\Sigma_{g}-N_{\epsilon}(p_{i})\hookrightarrow\Sigma_g^n$, where $N_{\epsilon}(p_{i})=\{x\in\Sigma_{g}^n: d(x,p_i)<\epsilon\}$ for a small $\epsilon>0$. Notice that $X_i\simeq \Sigma_{g,1}^{n-1}$. The kernel is generated by a Dehn twist around the boundary component, which is the simple loop $\partial N_{\epsilon}(p_{i})$ around the puncture $p_i$ in $\Sigma_{g}^n$. Observe that a permutation of the punctures induces a corresponding permutation of the surfaces $\{X_1,\ldots,X_n\}$, hence of the classes $\tau_i$ in $H^2(\PMod_g^n;\mathbb{Q})$.

We can also think of $\tau_i$ as the first Chern class of the line bundle $\mathbb{L}_i$ over $\mathcal{M}_{g,n}$ defined as follows: at a point in $\mathcal{M}_{g,n}$, i.e. a Riemann surface $X$ with marked points $p_1,\ldots,p_n$, the fiber of $\mathbb{L}_i$ is the cotangent space to $X$ at $p_i$. In fact, the $\tau$-classes are the image of the $\psi$-classes under the surjective homomorphism $H^2(\overline{\mathcal{M}}_{g,n};\mathbb{Q}) \rightarrow H^2(\mathcal{M}_{g,n};\mathbb{Q})$, where $\overline{\mathcal{M}}_{g,n}$ is the Deligne-Mumford compactification of $\mathcal{M}_{g,n}$  (see \cite{HAINLOO}). A permutation of the marked points induces the same permutation of the classes $\tau_i$ in $H^2(\mathcal{M}_{g,n};\mathbb{Q})$. 
Therefore, $S_n$ acts on the summand $\mathbb{Q}^{n}$ in (\ref{secondcoho}) by permuting the generators.

Thus, for $g\geq 4$ and $n\geq 3$, the decomposition of (\ref{secondcoho}) into irreducibles is given by $$H^2(\PMod_{g}^n;\mathbb{Q})\approx V(0)_n\oplus V(0)_n\oplus V(1)_n,$$
where, following our notation from Section \ref{Subsec1}, $V(0)_n$ is the trivial $S_n$-representation and $V(1)_n$ is the standard $S_n$-representation.
Notice that, even though the dimension of $H^2(\PMod_{g}^n;\mathbb{Q})$ blows up as $n$ increases, the decomposition into irreducibles stabilizes. In terms of definition of representation stability stated in  Section \ref{Subsec1}, we have shown that the sequence of $S_n$-representations $\{H^2(\PMod_g^n;\mathbb{Q})\}$ is {\it uniformly multiplicity stable} with stable range $n\geq 3$.  This indicates to us that representation stability of the cohomology groups of $\PMod_{g}^n$ may be the phenomena to expect.

\section{Representation stability for $H^*(G;V_n)$}\label{Coeff}
We discuss here when representation stability for a sequence $\{V_n\}$ of $G$-modules will imply representation stability for the cohomology of a group $G$  with coefficients $V_n$. This is Theorem \ref{GSTAB} below and it is a key ingredient for the base of the induction in the proof of Theorem \ref{MAIN}.

\begin{definition}
Let $G$ be a group. We will say that a sequence of rational vector spaces $V_n$ with given maps $\phi_n\colon V_n\rightarrow V_{n+1}$ is {\it consistent} and {\it compatible with $G$-actions} if it satisfies the following:
\begin{itemize}
 \item[-]{\textbf{Consistent Sequence.} Each $V_n$ is a rational $S_n$-representation and the map $\phi_n\colon V_n\rightarrow V_{n+1}$ is equivariant with respect to the inclusion $S_n\hookrightarrow S_{n+1}$.} 
\item[-]{\textbf{Compatible with $G$-actions.} Each $V_n$ is a $G$-module and the maps $\phi_n\colon V_n\rightarrow V_{n+1}$ are $G$-maps. The $G$-action commutes with the $S_n$-action.}
\end{itemize}
\end{definition}

Notice that for a sequence as in the previous definition and $p\geq 0$,  we have that $\{H^p(G;V_n);\phi_n^*\}$ is a consistent sequence of rational $S_n$-representations. Here $$\phi_n^*\colon H^p(G;V_n)\rightarrow H^p(G;V_{n+1})$$ denotes the map induced by $\phi_n\colon V_n\rightarrow V_{n+1}$. 

\begin{rtheo2}[Representation stability with changing coefficients]
 Let $G$ be a group of type $FP_{\infty}$. Consider a consistent sequence $\{V_n,\phi_n\}_{n=1}^{\infty}$ of finite dimensional rational representations of $S_n$ compatible with $G$-actions. If the sequence $\{V_n,\phi_n\}_{n=1}^{\infty}$ is monotone and uniformly representation stable with stable range $n\geq N$, then for any non-negative integer $p$, the sequence $\{H^p(G;V_n), \phi_n^*\}_{n=1}^{\infty}$ is monotone and uniformly representation stable with the same stable range.
\end{rtheo2}

\begin{proof}
Take $E\rightarrow \mathbb{Z}$ a free resolution of $\mathbb{Z}$ over $\mathbb{Z}G$  of finite type. This means that each $E_p$ is a free $G$-module of finite rank, say $E_p\approx (\mathbb{Z}G)^{d_p}$ generated by $x_1,\ldots,x_{d_p}$. 

There is an $S_n$-action on the chain complex $\mathcal{H}om(E,V_n)$ given by $\sigma \cdot h\colon x \mapsto \sigma \cdot h(x)$ for any $h\in\mathcal{H}om(E,V_n)$ and $\sigma \in S_n$. Since the $S_n$-action and the $G$-action on $V_n$ commute, this action restricts to a well-defined $S_n$-action on $\mathcal{H}om_G(E,V_n)$ which makes each $\mathcal{H}om_G(E,V_n)^p:=Hom_G(E_p,V_n)$ into a rational $S_n$-representation.

Observe that any $G$-homomorphism $h\colon E_p\rightarrow V_n$ is completely determined by the $d_p$-tuple $(h(x_1),\ldots,h(x_{d_p}))$. Then the assignment $h\mapsto (h(x_1),\ldots,h(x_{d_p}))$ gives us an isomorphism $$\mathcal{H}om_G(E,V_n)^p\approx V_n^{\oplus d_p}$$ not just of rational vector spaces, but of $S_n$-representations. Notice that since $V_n$ is finite dimensional, $\mathcal{H}om_G(E,V_n)^p$ also has finite dimension.
Moreover, under this isomorphism the map $$\phi_n^p:=\mathcal{H}om_G(E,\phi_n)^p\colon \mathcal{H}om_G(E,V_n)^p\rightarrow \mathcal{H}om_G(E,V_{n+1})^p$$ is just $(\phi_n)^{\oplus d_p}\colon V_n^{\oplus d_p}\rightarrow V_{n+1}^{\oplus d_p}$.
From Proposition \ref{SUB_QUO}, it follows that the sequence $\{\mathcal{H}om_G(E,V_n)^p;\phi_n^p\}$ is monotone and uniformly representation stable for $n\geq N$. 

The differentials $\delta_p^n$ of the cochain complex $\mathcal{H}om_G(E,V_n)$ are a consistent sequence of maps, meaning that the following diagram commutes: 
\begin{equation*}
\xymatrix{
\mathcal{H}om_G(E,V_n)^p\ar[d]_{\delta_p^n}\ar[r]^{\phi_n^p}&\mathcal{H}om_G(E,V_{n+1})^p\ar[d]_{\delta_p^{n+1}}\\
\mathcal{H}om_G(E,V_n)^{p+1}\ar[r]^{\phi_n^{p+1}}&\mathcal{H}om_G(E,V_{n+1})^{p+1}}
\end{equation*}
From Proposition \ref{KER_IM} the subsequences $\{\text{ker } \delta_p^n \}$ and $\{\text{im } \delta_p^n \}$ are monotone and uniformly representation stable for $n\geq N$. Finally Proposition \ref{SUB_QUO} gives the desired result for $H^p(G;V_n):= \text{ker } \delta_p^n / \text{im }\delta_{p+1}^n$.
\end{proof}

Since $H^0(G;V_n)$ is equal to the $G$-invariants $V_n^G$, as a particular case of Theorem  \ref{GSTAB}, we get the following.
\begin{cor}
The sequence of $G$-invariants $\{V_n^G,\phi_n\}$ is monotone and uniformly representation stable with the same stable range as $\{V_n,\phi_n\}$.\end{cor}

\section{Representation stability for $H^*(\PMod_g^n;\mathbb{Q})$}\label{SectionMain} In this section we prove our main result Theorem \ref{MAIN} and some consequences of it. We will focus on the sequence of pure mapping class groups $\PMod_g^n$ and its cohomology with rational coefficients. We consider the case $g\geq 2$. 
\subsection {The ingredients for the proof of the main theorem}\label{Ingred} Here we describe three of the four main ingredients needed in our proof of Theorem \ref{MAIN} in Section \ref{proof}. The ingredient (iv) is Theorem \ref{CHURCH_RS} \cite[Theorem 1]{CHURCH}.
 
\begin{large}{\bf (i) The Birman exact sequence}\end{large}.\label{S_nAction}
Our approach relies on the existence of a nice short exact sequence, introduced by Birman in 1969, that relates the pure mapping class group with the pure braid group of the surface: the \textit{Birman exact sequence} ($Bir1_n$). 

Let $C_n(\Sigma_g)$ be the configuration space of $\Sigma_g$ and $\mathfrak{p}=(p_1,\cdots,p_n)\in C_n(\Sigma_g)$ the punctures or marked points in $\Sigma_g^n$. The map in ($Bir1_n$) that realizes $\pi_1(C_n(\Sigma_g),\mathfrak{p})$ as a subgroup of $\PMod_g^n$ is the {\it point-pushing map} $\text{Push}$. For an element $\gamma\in\pi_1(C_n(\Sigma_g), \mathfrak{p})$, consider the isotopy defined by ``pushing'' the $n$-tuple $(p_1,\cdots,p_n)$ along $\gamma$. Then $\Push(\gamma)$ is represented by the  diffeomorphism at the end of the isotopy.
The map $f$ in $(Bir1_n)$ is a {\it forgetful morphism} induced by the inclusion $\Sigma_g^n\hookrightarrow \Sigma_g$. 

Taking the quotient  $(Bir1_n)$ by the $S_n$-action there, we obtain the Birman exact sequence $(Bir2_n)$. The relation between these two sequences is illustrated in the following diagram. 
\begin{equation}\label{D1}
\xymatrix{
&1\ar[d]& 1\ar[d]&&&\\
1\ar[r]& \pi_1(C_n(\Sigma_g))\ar[d]_q\ar[r]^{\text{Push}} &\PMod_g^n \ar[d]_q\ar[r]^f &\Mod_g \ar@{=}[d]_{id}\ar[r] &1&\text{   ($Bir1_n$)}\\
1\ar[r]& \pi_1(B_n(\Sigma_g))\ar[d]\ar[r]^{\text{Push}} &\Mod_g^n \ar[d]\ar[r]^f &\Mod_g \ar[r] &1&\text{   ($Bir2_n$)}\\
&S_n\ar[d]\ar@{=}[r]^{id}& S_n\ar[d]&&\\
&1&1&&}
\end{equation}

The columns in this diagram relate the groups $\pi_1(C_n(\Sigma_g))$ and $\PMod_g^n$ with  $\pi_1(B_n(\Sigma_g))$ and $\Mod_g^n$, respectively, in the same way as the pure braid group $P_n$ is related to the braid group $B_n$ by the short exact sequence $$1\rightarrow P_n\rightarrow B_n\rightarrow S_n\rightarrow 1.$$ Proofs of the exactness of the sequences in diagram (\ref{D1}) can be found in \cite{BIRMAN} and \cite{FARBMARG}. The exactness of $(Bir1_1)$ and $(Bir2_n)$ requires $g\geq 2$.

Observe that from the short exact sequence ($Bir1_n$) we get an action of $\Mod_g$ on $H^*(\pi_1(C_n(\Sigma_g));\mathbb{Q})$. The second column in diagram (\ref{D1}) defines an $S_n$- action on $H^*(\PMod_g^n;\mathbb{Q})$ which restricts to the $S_n$-action on $H^*(\pi_1(C_n(\Sigma_g));\mathbb{Q})$ defined by the short exact sequence in the first column. The induced map $$\text{Push}^*\colon H^*(\PMod_g^n;\mathbb{Q})\rightarrow H^*(\pi_1(C_n(\Sigma_g));\mathbb{Q})$$ is a $S_n$-map between rational $S_n$-representations. Moreover, from the commutativity of diagram (\ref{D1}) we have the following.

\begin{prop}\label{COMMUTE}
The actions of $S_n$ and $\Mod_g$ on  $H^*(\pi_1(C_n(\Sigma_g));\mathbb{Q})$ commute.
\end{prop}

\begin{large}{\bf (ii) The Hochschild-Serre spectral sequence}\end{large}.
We denote the Hochschild-Serre spectral sequence associated to the short exact sequence ($Bir1_n$) by 
$E_*(n)$, where the $E_2$-page is given by:

\begin{equation*}
E_2^{p,q}(n)=H^p\big(\Mod_g;H^q(\pi_1(C_n(\Sigma_g));\mathbb{Q})\big),
\end{equation*}

and the spectral sequence converges to $H^{p+q}(\PMod_g^n;\mathbb{Q})$. This spectral sequence gives a natural filtration of $H^i(\PMod_g^n;\mathbb{Q})$:
\begin{equation} \label{FILT}
0\leq F^i_i(n)\leq F^i_{i-1}(n)\leq \cdots\leq F^i_1(n)\leq F^i_0(n)=H^i(\PMod_g^n;\mathbb{Q}),
\end{equation}
where the successive quotients are $F^i_p(n)/F^i_{p+1}(n)\cong E_{\infty}^{p,i-p}(n)$.
\bigskip

The following lemma is due to Harer (\cite[Theorem 4.1]{HARERVCD}) and establishes that $\Mod_g$ satisfies the finiteness conditions that our argument requires.

\begin{lemma}\label{HARER}
For $2g+s+r>2$, the mapping class group $\Mod_{g,r}^s$ is a virtual duality group with virtual cohomological dimension $d(g,r,s)$, where $d(g,0,0)=4g-5$, $d(g,r,s)=4g+2r+s-4$, $g>0$ and $r+s>0$, and $d(0,r,s)=2r+s-3$. In particular, $\Mod_{g,r}^s$ is a group of type $FP_{\infty}$, and for any rational $\Mod_{g,r}^s$-module $M$, we have $H^p(\Mod_{g,r}^s;M)=0$ for $p>d(g,r,s)$.
\end{lemma}

We now see that the terms of the spectral sequence $E_*(n)$ are finite dimensional $S_n$-representations.

\begin{prop}\label{S_nSPECTRAL}
For $2\leq r\leq\infty$, each $E_r^{p,q}(n)$ is a finite dimensional rational $S_n$-representation and the differentials $$d_r^{p,q}(n)\colon E_r^{p,q}(n)\rightarrow E_r^{p+r,q-r+1}(n)$$ are $S_n$-maps.
\end{prop}

\begin{proof}
Let $\sigma\in S_n$ and take $\tilde{\sigma}\in\text{Push}(\pi_1(B_n(\Sigma_g))<\Mod_g^n$ (see ($Bir2_n$)). Denote by $c(\tilde{\sigma})$ the conjugation by $\tilde{\sigma}$. Diagram (\ref{D1}) then gives:

\begin{equation*}
\xymatrix{
1\ar[r]& \pi_1(C_n(\Sigma_g))\ar[d]_{c(\tilde{\sigma})}\ar[r] &\PMod_g^n \ar[d]_{c(\tilde{\sigma})}\ar[r] &\Mod_g \ar@{=}[d]_{id}\ar[r] &1\\
1\ar[r]& \pi_1(C_n(\Sigma_g))\ar[r] &\PMod_g^n \ar[r] &\Mod_g \ar[r] &1}
\end{equation*}

The induced maps $c(\tilde{\sigma})^*_r\colon E_r^{p,q}(n)\rightarrow E_r^{p,q}(n)$ do not depend on the lift of $\sigma\in S_n$ and, by naturality of the Hochschild-Serre spectral sequence, they commute with the differentials. Hence we get an $S_n$-action on each $E_r^{p,q}(n)$ for $2\leq r\leq\infty$ that commutes with the differentials. Moreover, naturality also implies that the $S_n$-action on $H^*(\PMod_g^n;\mathbb{Q})$ induces the corresponding $S_n$-action on $E_{\infty}^{p,q}(n)$.

By Lemma \ref{HARER}, the group $\Mod_g$ is of type $FP_{\infty}$. Totaro showed in \cite[Theorem 4]{TOTARO} that the cohomology ring  $H^*(\pi_1(C_n(\Sigma_g));\mathbb{Q})$ is generated by cohomology classes from the rings $H^*(\Sigma_g;\mathbb{Q})$ and $H^*(P_n;\mathbb{Q})$. In particular, his result implies that $H^q(\pi_1(C_n(\Sigma_g));\mathbb{Q})$ is a finite dimensional $\mathbb{Q}$-vector space for $q\geq 0$. It follows that $$E_2^{p,q}(n)=H^p\big(\Mod_g;H^q(\pi_1(C_n(\Sigma_g));\mathbb{Q})\big)$$ is a finite dimensional $\mathbb{Q}$-vector space, and likewise for the subquotients $E_r^{p,q}(n)$.
\end{proof}

\begin{large}{\bf (iii) The forgetful map}\end{large}.
For the pure braid group, there is a natural map  $f_n\colon P_{n+1}\rightarrow P_{n}$ given by ``forgetting'' the last strand. Similarly, the inclusion $\Sigma_g^{n+1}\hookrightarrow \Sigma_g^n$ induces a homomorphism $$f_n\colon \PMod_g^{n+1} \rightarrow \PMod_g^{n} $$ that we call the \textit{forgetful map}. We can also think of this map as the one induced by ``forgetting a marked point'' in $\Sigma_g^n$.  When restricted to the subgroup $\text{Push}\big(\pi_1(C_{n+1}(\Sigma_g))\big)$ it corresponds to the homomorphism in fundamental groups induced by the map $C_{n+1}(\Sigma_g)\rightarrow C_{n}(\Sigma_g)$ given by ``forgetting the last coordinate''. This gives rise to the commutative diagram (3) that relates the exact sequences $(Bir1_{n+1})$ and $(Bir1_{n})$.

\begin{equation}\label{D2}
\xymatrix{
&1\ar[d]& 1\ar[d]&&&\\
&\pi_1(\Sigma_g^n)\ar[d]\ar@{=}[r]^{id}&\pi_1(\Sigma_g^n)\ar[d]&&\\
1\ar[r]& \pi_1(C_{n+1}(\Sigma_g))\ar[d]_{f_n}\ar[r] &\PMod_g^{n+1} \ar[d]_{f_n}\ar[r] &\Mod_g \ar@{=}[d]_{id}\ar[r] &1&\text{ ($Bir_{n+1}$)}\\
1\ar[r]& \pi_1(C_n(\Sigma_g))\ar[d]\ar[r] &\PMod_g^n \ar[d]\ar[r] &\Mod_g \ar[r] &1&\text{ ($Bir1_n$)}\\
&1&1&&\\}
%&\text{(C)}&\text{(D)}&&&}
\end{equation}
Diagram (\ref{D2}) and our remarks in Section \ref{Subsec2} imply the following.
\begin{prop}\label{G-map}
 The induced maps $f_n^*\colon H^*(\pi_1(C_n(\Sigma_g));\mathbb{Q})\rightarrow H^*(\pi_1(C_{n+1}(\Sigma_g);\mathbb{Q})$ are $\Mod_g$-maps.
\end{prop}

Moreover, diagram (\ref{D2}) and naturality of the Hochschild-Serre spectral sequence give us:

\begin{itemize}
 \item [1)] {Induced maps $(f_n)^*_r\colon E_r^{p,q}(n)\rightarrow E_r^{p,q}(n+1)$ that commute with the differentials. This means that the differentials $d_r^{p,q}(n)$ are consistent maps in the sense of Proposition \ref{KER_IM}}. 
\item [2)] {The map $(f_n)^*\colon H^*(\PMod_g^n;\mathbb{Q})\rightarrow H^*(\PMod_g^{n+1};\mathbb{Q})$ preserves the filtrations (\ref{FILT}) inducing a map on the succesive quotients $E_{\infty}^{p,q}(n)$ which is the map
$(f_n)^*_{\infty}\colon E_{\infty}^{p,q}(n)\rightarrow E_{\infty}^{p,q}(n+1)$.}
\item [3)] {The map $(f_n)^*_2\colon E_2^{p,q}(n)\rightarrow E_2^{p,q}(n+1)$ is the one induced by the group homomorphisms id$\colon \Mod_g\rightarrow \Mod_g$ and $f_n\colon \pi_1(C_{n+1}(\Sigma_g))\rightarrow \pi_1(C_n(\Sigma_g))$.}
\end{itemize}

\subsection{The proof of the main theorem (Theorem \ref{MAIN})} \label{proof}
In order to prove Theorem \ref{MAIN} we use an inductive argument on the pages of the spectral sequence described in Section \ref{Ingred} (ii). The following lemma gives us the base of the induction.

\begin{lemma}\label{E2} For each $p\geq 0$ and $q\geq 0$, the consistent sequence of rational $S_n$-representations $$\{ E_2^{p,q}(n)=H^p\big(\Mod_g;H^q(\pi_1(C_n(\Sigma_g));\mathbb{Q})\big)\}$$ is monotone and uniformly representation stable with stable range $n\geq 4q$. 
\end{lemma}

\begin{proof}
Let $q\geq 0$. Since $C_n(\Sigma_g)$ is aspherical, by Theorem \ref{CHURCH_RS} of Church we have that the consistent sequence of rational $S_n$-representations $\{H^q(\pi_1(C_n(\Sigma_g));\mathbb{Q})\}$ with the forgetful maps $$f_n\colon H^q(\pi_1(C_n(\Sigma_g));\mathbb{Q})\rightarrow  H^q(\pi_1(C_{n+1}(\Sigma_g));\mathbb{Q})$$ is monotone and uniformly representation stable with stable range $n\geq 4q$. Moreover, Propositions \ref{COMMUTE} and \ref{G-map} imply that the sequence is compatible with the $\Mod_g$-action. The group $\Mod_g$ is $FP_{\infty}$ (Lemma \ref{HARER}). Hence we can apply Theorem \ref{GSTAB}.
\end{proof}

From Lemma \ref{E2}, we follow the same type of inductive argument from \cite[Section 3]{CHURCH} that Church uses in order to prove his main result \cite[Theorem 1]{CHURCH}. Here we get monotonicity and uniform representation stability for all the pages of the spectral sequence $E_*(n)$. We include the proofs here for completeness.

\begin{lemma}\label{Er}
The sequence $\{E_{r}^{p,q}(n)\}$ is monotone and uniformly representation stable with stable range $n\geq 4q+2(r-1)(r-2)$.
\end{lemma}
\begin{proof}
The proof is done by induction on $r$ where the base case $r=2$ is given by Lemma \ref{E2}.
Assume that $\{E_r^{p,q}(n)\}$ is monotone and uniformly representation stable for $n\geq 4(q+\sum_{k=1}^{r-2} k)$.

As noted before, the differentials $$d_r^{p,q}(n)\colon E_r^{p,q}(n)\rightarrow E_{r}^{p+r,q-r+1}(n)$$ are a consistent sequence of maps in the sense of Proposition \ref{KER_IM}. Then $\{\text{ker } d_r^{p,q}(n)\}$ is monotone and uniformly representation stable for $n\geq 4(q+\sum_{k=1}^{r-2} k)$. Moreover $\{\text{im } d_r^{p-r,q+r-1}(n)\}$ is monotone and uniformly representation stable for $n\geq 4(q+(r-1)+\sum_{k=1}^{r-2} k)$. Therefore by Proposition \ref{SUB_QUO} the next page in the spectral sequence
$$E_r^{p,q}(n)\cong \text{ker } d_r^{p,q}(n)/ \text{im } d_r^{p-r,q+r-1}$$
is monotone and uniformly representation stable for $n\geq 4(q+\sum_{k=1}^{r-1} k)$.
\end{proof}

\begin{lemma}\label{Einf}
For every $p,q\geq 0$ and every $n\geq 2$, we have $E_{\infty}^{p,q}(n)=E_{R}^{p,q}(n)$, where $$R=4g-4=\text{vcd}(\Mod_g)+1.$$
\end{lemma}
\begin{proof}
The Hochschild-Serre spectral sequence $E_*(n)$ is a first-quadrant spectral sequence. Moreover, from Lemma \ref{HARER} it follows that for every $p>4g-5$ $$0=H^p(\Mod_g;H^q(\pi_1(C_n(\Sigma_g))=E_2^{p,q}(n)=E_r^{p,q}(n).$$  Therefore for $R=4g-4$, $q\geq 0$ and $0\leq p\leq 4g-5$, we have that $E_R^{p-R,q+R-1}(n)=0$ since $p-R<0$ and $E_R^{p+R,q-R+1}(n)=0$ since $p+R>4g-5$. Then the differentials $d_R^{p,q}$ and $d_R^{p-R,q+R-1}$ are zero and hence
$$E_{R+1}^{p,q}(n)=\text{ ker } d_R^{p,q} / \text{im } d_R^{p-R,q+R-1}=E_{R}^{p,q}(n).$$
\end{proof}

Having built up, we are now able to prove our main result: uniform representation stability of $\{H^i(\PMod_{g,r}^n;\mathbb{Q})\}_{n=1}^{\infty}$.

\begin{rtheo1} 
 For any $i\geq 0$ and $g\geq 2$ the sequence of cohomology groups $\{H^i(\PMod_g^n;\mathbb{Q})\}_{n=1}^{\infty}$ is monotone and uniformly representation stable with stable range $$n\geq\text{min } \{4i+2(4g-6)(4g-5), 2i^2+6i\}.$$
\end{rtheo1}
\begin{proof}
Each of the successive quotients of the natural filtration (\ref{FILT}) of $H^i(\PMod_g^n;\mathbb{Q})$ give us a sequence $$\{F^i_p(n)/F^i_{p+1}(n)\approx E_{\infty}^{p,i-p}(n)\}$$ which, by Lemmas \ref{Er} and \ref{Einf}, is monotone and uniformly representation stable with stable range $n\geq 4(i-p)+2(4g-6)(4g-5)$. This is the case, in particular, for $F^i_{i-1}(n)/F^i_{i}(n)$ and $F^i_{i}(n)\approx E_{\infty}^{i,0}(n)$. Then by Proposition \ref{SUB_QUO} we have that $F^i_{i-1}(n)$ is monotone and uniformly representation stable. Reverse induction and Proposition \ref{SUB_QUO} imply that the sequences $\{F^i_p(n)\}$ ($0\leq p\leq i$) are monotone and uniformly representation stable with the same stable range. In particular this is true for $F^i_0(n)=H^i(\PMod_g^n;\mathbb{Q})$.

Observe that $$ 4(i-p)+2(4g-6)(4g-5)+4p\geq 4(i-p)+2(4g-6)(4g-5)$$ for all $0\leq p\leq i$, which give us the desired stable range.

Finally, we notice that for a fixed $i\geq 0$, the group $H^i(\PMod_g^n;\mathbb{Q})$ only depends on the terms $E_{\infty}^{p,i-p}(n)=E_{i+2}^{p,i-p}(n)$, $i\geq p\geq 0$. Hence from Lemma \ref{Er} we get a stable range that does not depend on the genus $g$. However, this stable range is quadratic on $i$: the sequence $\{H^i(\PMod_g^n;\mathbb{Q})\}$ is monotone and uniformly representation stable for $n\geq 4i+2(i+1)(i)=(2i)(i+3)$.
\end{proof}

\subsection{Rational homological stability for $\Mod_g^n$}\label{Conseq}
From the short exact sequence in the second column of diagram (1), we have that any rational $S_n$-representation can be regarded as a representation of $\Mod_g^n$ by composing with the projection $\Mod_g^n\rightarrow S_n$. As a consequence of Theorem $\ref{MAIN}$ we get cohomological stability for $\Mod_g^n$ with twisted coefficients. 

\begin{rcor4}
 For any  partition $\lambda$ , the sequence $\{H^i\big(\Mod_g^n;V(\lambda)_n\big)\}_{n=1}^{\infty}$ of twisted cohomology groups satisfies classical cohomological stability: for fixed $i\geq 0$ and $g\geq 2$, there is an isomorphism  $$H^i\big(\Mod_g^n;V(\lambda)_n\big)\approx H^i\big(\Mod_g^{n+1};V(\lambda)_{n+1}\big),$$ if $n\geq \text{min }\{4i+2(4g-6)(4g-5), n\geq 2i^2+6i\}$.
\end{rcor4}
\begin{proof} 
This is just the argument by Church-Farb in \cite[Corollary 4.4]{CHURCH_FARB}. The group $\PMod_g^n$ is a finite index subgroup of $\Mod_g^n$ and the coefficients $V(\lambda)_n$ are rational vector spaces, therefore the transfer map (see \cite{BROWN}) give us an isomorphism   $$H^i\big(\Mod_g^n;V(\lambda)_n\big)\approx H^i\big(\PMod_g^n;V(\lambda)_n\big)^{S_n}.$$
Moreover, $V(\lambda)_n$ is a trivial $\PMod_g^n$-representation, since the action of $\Mod_g^n$ on $V(\lambda)_n$ factors through $S_n$. Hence, from the universal coefficient theorem, we have  
\begin{equation}\label{INV}
H^i\big(\PMod_g^n;V(\lambda)_n)\big)^{S_n}\approx \Big(H^i(\PMod_g^n;\mathbb{Q})\otimes V(\lambda)_n\Big)^{S_n}.
\end{equation}
For two partitions $\lambda$ and $\mu$ of $n$ the representation $V(\lambda)\otimes V(\mu)$ contains the trivial representation if and only if $\lambda=\mu$, in which case it has multiplicity 1 (see \cite{FULTON_HARRIS}) . Therefore the dimension of (\ref{INV}) is the multiplicity of $V(\lambda)_n$ in $H^i(\PMod_g^n;\mathbb{Q})$ which is constant for  $n\geq 4i+2(4g-6)(4g-5)$ by Theorem \ref{MAIN}.
\end{proof}

In particular, the multiplicity of the trivial representation in $H^i(\PMod_g^n;\mathbb{Q})$, which equals $H^i(\Mod_g^n;\mathbb{Q})$, is constant for $n\geq 4i+2(4g-6)(4g-5)$. In fact, the stable range in this case can be slightly improved.

\begin{cor}\label{CLASSTAB}
For any $i\geq 0$ and a fixed $g\geq 2$, the sequence of mapping class groups $\{\Mod_g^n\}_{n=1}^{\infty}$ satisfies rational cohomological stability:  
\begin{equation*}
H^i(\Mod_g^n;\mathbb{Q})\approx H^i(\Mod_g^{n+1};\mathbb{Q}),
\end{equation*}
if $n\geq \text{max } \{i+(2g-3)(4g-5), 2i^2+4i\}$.
\end{cor}

\begin{proof}
For any $n$ the $S_n$-invariants of the spectral sequence  $(E_2^{p.q})^{S_n}$ form a spectral sequence that converges to $H^{p+q}(\PMod_g^n;\mathbb{Q})^{S_n}$. In fact, $(E_2^{p.q})^{S_n}$  is just the $(p,q)$-term of the $E_2$-page of the Hochschild-Serre spectral sequence of the group extension $(Bir2_n)$ converging to  $H^{p+q}(\Mod_g^n;\mathbb{Q})$.
In \cite[Corollary 3]{CHURCH} a better stable range than the one in Theorem \ref{CHURCH_RS} is obtained when restricted to the $S_n$-invariants: the dimension of $H_q(C_n(\Sigma_g);\mathbb{Q})^{S_n}$ is constant for $n> q$.  As a consequence the dimension of  $(E_2^{p.q})^{S_n}$ is constant for $n\geq q$. 
Proposition \ref{PARTITION} allows us to repeat the general argument for this spectral sequence of $S_n$-invariants in order to get the desired stable range. 
\end{proof}

\subsection{Non-closed surfaces}

Our main result is also true if we consider a non-closed surface $\Sigma_{g,r}^s$ of genus $g$, with $r$ boundary components and $s$ punctures with $2g+r+s>2$. 

Let $p_1,\ldots, p_n$ be distinct points in the interior of  $\Sigma_{g,r}^s$. We define the {\it mapping class group}  $\Mod^n(\Sigma_{g,r}^s)$ as the group of isotopy classes of orientation-preserving self-diffeomorphisms of $\Sigma_{g,r}^s$ that permute the distinguished points $p_1,\ldots, p_n$ and that restrict to the identity on the boundary components. The {\it pure mapping class group} $\PMod_{g,r}^n$ is defined analogously by asking that the distinguished points $p_1,\ldots, p_n$ remain fixed pointwise.

When $2g+r+s>2$ we have again a Birman exact sequence (see \cite{FARBMARG}):
$$1\rightarrow \pi_1(C_n(\Sigma_g^{r+s}))\rightarrow \PMod^n(\Sigma_{g,r}^s)\rightarrow \Mod_{g,r}^s\rightarrow 1.$$
In particular, this includes the three punctured sphere $\Sigma_0^{3}$ and the punctured torus $\Sigma_1^{1}$.

Using this short exact sequence and Theorem \ref{CHURCH_RS} we can use the previous arguments to get representation stability for the cohomology of $\PMod^n(\Sigma_{g,r}^s)$, when $2g+s+r>2$.
\begin{theo}\label{BOUNDARY}
 For any $i\geq 0$ and  $2g+s+r>2$ the sequence  $\{H^i(\PMod^n(\Sigma_{g,r}^s);\mathbb{Q})\}_{n=1}^{\infty}$ is monotone and uniformly representation stable with stable range $$n\geq\text{min } \{4i+2\big(d(g,r,s)\big)\big(d(g,r,s)-1\big), 2i^2+6i\}.$$

\noindent Furthermore for any  partition $\lambda$ and any fixed $i\geq 0$ and  $2g+s+r>2$, there is an isomorphism  $$H^i\big(\Mod^n(\Sigma_{g,r}^s);V(\lambda)_n\big)\approx H^i\big(\Mod^{n+1}(\Sigma_{g,r}^s);V(\lambda)_{n+1}\big),$$ if $n\geq\text{min } \{4i+2\big(d(g,r,s)\big)\big(d(g,r,s)-1\big), 2i^2+6i\}.$
\end{theo}

Here $d(g,r,s)$ denotes the virtual cohomological dimension of $\Mod_{g,r}^s$ as in Lemma \ref{HARER}.

In the case of trivial coefficients $V(0)_n=\mathbb{Q}$ we recover puncture stability for the rational cohomology groups of $\Mod^n(\Sigma_{g,r}^s)$ for  $2g+s+r>2$. 

\section{Pure mapping class groups of higher dimensional manifolds}\label{Section5}
We now explain how the key ideas from before can be applied to obtain representation stability for the cohomology of pure mapping class groups of higher dimensional manifolds.

\subsection{Representation stability for $H^*(\PMod^n(M);\mathbb{Q})$} \label {Mod_M}
Let $M$ be a connected, smooth manifold and consider  the mapping class group $\Mod^n(M)$ and the pure mapping class group $\PMod^n(M)$ as defined in the introduction. We now show how, in some cases, the previous techniques and Proposition \ref{PRODUCT} from Section \ref{Subsec5} can be used to prove representation stability for $\{H^i(\PMod^n(M);\mathbb{Q}), f_n^i\}$.\bigskip

{\bf Notation: }We denote by $C_n(M)$ (resp. $B_n(M)$) the configuration space of $n$ distinct ordered (resp. unordered) points in the interior of any manifold $M$. We refer to $p_1,\ldots, p_n$ as the ``punctures'' or the ``marked points''. We will usually take the $n$-tuple $\mathfrak{p}=(p_1,\ldots, p_n)\in C_n(M)$  as the base point of $\pi_1(C_n(M))$ (resp. $\pi_1(B_n(M))$). The group $P_n:=\pi_1(C_n(\mathbb{R}^2),\mathfrak{p})\approx\PMod_{0,1}^n$ is the {\it pure braid group} and the {\it braid group} is $B_n:=\pi_1(B_n(\mathbb{R}^2),\mathfrak{p})\approx\Mod_{0,1}^n$.

The inclusion $$\big(M-\{p_1,\ldots, p_n, p_{n+1}\}\big)\hookrightarrow \big( M-\{p_1,\ldots, p_n\}\big)$$ induces the \textit{forgetful homomorphism}  $$f_n\colon \PMod^{n+1}(M)\rightarrow \PMod^{n}(M).$$

Recall that one of the main ingredients needed in our proof of Theorem \ref{MAIN} is the existence of a Birman exact sequence that allows us to relate $\pi_1(C_n(M),\mathfrak{p})$ with $\PMod^n(M)$. First we notice that, when the dimension of $M$ is $d\geq 3$, the group $\pi_1(C_n(M))$ can be completely understood in terms of $\pi_1(M)$.

\begin{lemma}\label{PRODBIR}
Let $M$ be a smooth connected manifold of dimension $d\geq 3$. Then for any $n\geq 1$ the inclusion map $C_n(M)\hookrightarrow M^n$ induces an isomorphism $\pi_1(C_n(M),\mathfrak{p})\approx \pi_1(M^n,\mathfrak{p})\approx \prod_{i=1}^n \pi_1(M,p_i)$.
\end{lemma}

The case for closed manifolds is due to Birman (\cite[Theorem 1]{BIRMAN1969}).  As Allen Hatcher explained to me, there are many manifolds for which there is a Birman exact sequence.

\begin{lemma}[Existence of a Birman Exact Sequence] \label{EXISTENCEBIR}
Let $M$ be a smooth connected manifold of dimension $d\geq 3$. If the fundamental group $\pi_1(M)$ has trivial center or $\Diff (M)$ is simply connected, then there exists a Birman exact sequence 
\begin{equation}\label{BIRDIM>3}
\xymatrix{
1\ar[r]& \pi_1(C_n(M))\ar[r] &\PMod^n (M)\ar[r] &\Mod (M) \ar[r] &1.}
\end{equation}
\end{lemma}

\begin{proof}
The evaluation map $$ev\colon \Diff(M)\rightarrow C_n(M),$$ given  by $ f\mapsto (f(p_1),\ldots,f(p_n))$ is a fibration with fiber $\PDiff^n(M)$. Consider the associated long exact sequence in homotopy groups 
\begin{equation*}
\xymatrix{
\cdots\ar[r]& \pi_1(\Diff(M))\ar[r] &\pi_1(C_n(M))\ar[r]^{\delta} &\pi_0(\PDiff^n (M))\ar[r] &\pi_0(\Diff(M)) \ar[r] &1.}
\end{equation*}
If $\Diff (M)$ is simply connected, then the existence of the short exact sequence (\ref{BIRDIM>3}) follows. On the other hand, we may consider the map $$\psi: \pi_0(\PDiff^n (M))\rightarrow \text{Aut}[\pi_1(C_n(M))]$$ given by $[f]\mapsto [\gamma\mapsto f\circ\gamma].$ 

The composition
\begin{equation*}
\xymatrix{
\pi_1(C_n(M))\ar[r]^{\delta} &\pi_0(\PDiff^n (M))\ar[r]^{\psi} &\Aut[\pi_1(C_n(M))]}
\end{equation*}
sends $\sigma\in\pi_1(C_n(M))$ to the inner automorphism $c(\sigma)$ given by conjugation by $\sigma$. 
If the dimension $d\geq 3$ and $\pi_1(M)$ has trivial center, then so does $\pi_1(C_n(M))$ by Lemma \ref{PRODBIR}. In this case, the boundary map $\delta$ is injective and we get the desired Birman exact sequence (\ref{BIRDIM>3}).
\end{proof}

The $E_2$-page of the Hochschild-Serre spectral sequence associated to (\ref{BIRDIM>3}) is then
$$ E_2^{p,q}(n)=H^p\big(\Mod (M);H^q(\pi_1(C_n(M));\mathbb{Q})\big).$$

By Lemma \ref{PRODBIR} 
$$H^q(\pi_1(C_n(M));\mathbb{Q}))= H^q(\pi_1(M)^n;\mathbb{Q}).$$  

Moreover, by Proposition \ref{PRODUCT} below, if the group $\pi_1(M)$ is of type $FP_{\infty}$,  the consistent sequence $\{H^q(\pi_1(M)^n;\mathbb{Q})\}_{n=1}^{\infty}$ is monotone and uniformly representation stable, with stable range $n\geq 2q$. Hence when $\Mod(M)$ is also of type $FP_{\infty}$ (e.g. $M$ is compact),  Theorem \ref{GSTAB} and the same inductive argument on the succesive pages of spectral sequence yield the following:

\begin{lemma}\label{E2DIM>3} For every $i\geq 0$ and every $n\geq 2$, the consistent sequence of rational $S_n$-representations $$\{ E_2^{i-q,q}(n)=H^{i-q}\big(\Mod(M);H^q(\pi_1(C_n(M));\mathbb{Q})\big)\}_{n=1}^{\infty}$$ is monotone and uniformly representation stable with stable range $n\geq 2q$. Furthermore $E_{\infty}^{i-q,q}(n)=E_{i+2}^{i-q,q}(n)$, which is monotone and uniformly representation stable with stable range $$n\geq 2q+2(i+1)(i).$$
\end{lemma}

Observe that now we have all the ingredients needed in order to reproduce our arguments from Section \ref{proof} and prove Theorem \ref{MAINDIM>3} and Corollary \ref{TWISTSTABDIM>3}.

\subsection{Representation stability of $H^*(G^n;\mathbb{Q})$}\label{Subsec5}
Given a group $G$, we may consider the sequence of groups $\{G^n=\prod_{i=1}^nG\}$ with the corresponding $S_n$-action given by permuting the factors. The natural homomorphism $G^{n+1}\rightarrow G^n$ by forgetting the last coordinate is equivariant with respect to the inclusion $S_n\hookrightarrow S_{n+1}$. For a fixed $q\geq 0$ the induced maps $$\phi_n\colon H^q(G^n;\mathbb{Q})\rightarrow H^q(G^{n+1};\mathbb{Q})$$ give us a consistent sequence of $S_n$-representations. If $G$ is of type $FP_{\infty}$, we have finite dimensional representations. Monotonicity and uniform representation stability of this sequence are a particular case of  \cite[Proposition 3.1]{CHURCH} (corresponding to the first row in the spectral sequence). Since this result gives us the inductive hypothesis for the proof of Theorem \ref{MAINDIM>3}, we present here a complete proof for the reader's convenience.
 
For a fixed $S_l$-representation $V$ and each $n\geq l$, we denote by $V_{\alpha}\boxtimes\mathbb{Q}$ the corresponding $(S_l\times S_{n-l})$-representation, where  the factor $S_{n-l}$ acts trivially. We can then consider the sequence of $S_n$-representation $\{\text{Ind}_{S_l\times S_{n-l}}^{S_n} V_{\alpha}\boxtimes\mathbb{Q}\}$ with the natural inclusions $$\iota_n\colon \text{Ind}_{S_l\times S_{n-l}}^{S_n} V_{\alpha}\boxtimes\mathbb{Q}\hookrightarrow \text{Ind}_{S_l\times S_{n+1-l}}^{S_{n+1}} V_{\alpha}\boxtimes\mathbb{Q}.$$ 
This sequence is monotone and uniform representation stable as proved in \cite[Theorem 2.11]{CHURCH}: 

\begin{lemma}\label{IND}
Let $V$ be a finite dimensional $S_l$-representation, then the sequence of induced representations $\{\text{Ind}_{S_l\times S_{n-l}}^{S_n} V\boxtimes\mathbb{Q}\}_{n=1}^{\infty}$ is monotone and uniformly representation stable for $n\geq 2l$. 
\end{lemma}

This lemma and the K\"{u}nneth formula give us the following result.

\begin{prop}\label{PRODUCT} Let $G$ be any group of type $FP_{\infty}$ and $q\geq 0$. The consistent sequence of $S_n$-representations
$\{H^q(G^n;\mathbb{Q}),\phi_n\}_{n=1}^{\infty}$ is monotone and uniformly representation stable for $n\geq 2q$.
\end{prop}
\begin{proof}
By the K\"{u}nneth formula we have $$H^q(G^n;\mathbb{Q})\approx\bigoplus_{\mathfrak{a}} H^{\mathfrak{a}}(G^n)$$ where the sum is over all tuples $\mathfrak{a}=(a_1,\ldots, a_n)$ such that $a_j\geq 0$ and $\sum a_j=q$ and $H^{\mathfrak{a}}(G^n)$ denotes $H^{a_1}(G;\mathbb{Q})\otimes\cdots\otimes H^{a_n}(G;\mathbb{Q}).$

Let $\overline{\mathfrak{a}}=\alpha$ where $\alpha=(\alpha_1\geq \alpha_2\geq\ldots\geq\alpha_l)$ is a partition of $q$ and the $\alpha_j$ are the positive values of $\mathfrak{a}$ arranged in decreasing order. We define $\text{supp}(\mathfrak{a})$ as the subset of $\{1,2,\ldots, n\}$ for which $a_i\neq0$. Observe that the length of $\alpha$ is $l=\vline \text{supp}(\mathfrak{a})\vline\leq q$. Therefore we have
$$H^q(G^n;\mathbb{Q})= \bigoplus_{\alpha} H^{\alpha}(G^n)$$
where now the sum is over all partitions $\alpha$ of $q$ of length $l\leq q$ and $H^{\alpha}(G^n)=\bigoplus_{\overline{\mathfrak{a}}=\alpha} H^{\mathfrak{a}}(G^n)$.

The natural $S_n$-action on $G^n$ induces an $S_n$-action on $H^q(G^n;\mathbb{Q})$. More precisely, the group $S_n$ acts on $n$-tuples $\mathfrak{a}$ by permuting the coordinates. This induces an action on $\bigoplus_{\overline{\mathfrak{a}}=\alpha} H^{\mathfrak{a}}(G^n)$ by permuting the summands accordingly (with a sign, since cohomology is graded commutative) . Hence, under this action, each  $H^{\alpha}(G^n)$ is $S_n$-invariant. We now describe  $H^{\alpha}(G^n)$ as an induced representation.

For a given $\alpha$, take $\mathfrak{b}=(\alpha_1,\ldots,\alpha_l,0,\cdots,0)$.  Observe that we can identify the $S_n$-translates of $H^{\mathfrak{b}}(G^n)$ with the cosets $S_n/\text{Stab}(\mathfrak{b})$ by an orbit-stabilizer argument. Thus
$$H^{\alpha}(G^n)=\text{Ind}_{\text{Stab}(\mathfrak{b})}^{S_n} H^{\mathfrak{b}}(G^n).$$
Moreover, $S_{n-l}<\text{Stab}(\mathfrak{b})<S_l\times S_{n-l}$, where $S_l$ permutes coordinates $\{1,\ldots ,l\}$ and $S_{n-l}$ permutes coordinates $\{l+1,\ldots, n\}$. Therefore $\text{Stab}(\mathfrak{b})=H\times S_{n-l}$, for some subgroup  $H<S_l$.

Notice that $$H^{\mathfrak{b}}(G^n)=H^{b_1}(G;\mathbb{Q})\otimes\cdots\otimes H^{b_l}(G;\mathbb{Q})\otimes\cdots\otimes H^0(G;\mathbb{Q})\approx H^{b_1}(G;\mathbb{Q})\otimes\cdots\otimes H^{b_l}(G;\mathbb{Q})$$ can be regarded as an $H$-representation. 

Let $V_{\alpha}:= \text{Ind}_{H}^{S_l} H^{\mathfrak{b}}(G^n)$ and let $V_{\alpha}\boxtimes\mathbb{Q}$ denote the corresponding $(S_l\times S_{n-l})$-representation. Then

$$\hspace{27 mm}H^{\alpha}(G^n)=\text{Ind}_{\text{Stab}(\mathfrak{b})}^{S_n} H^{\mathfrak{b}}(G^n)
=  \text{Ind}_{H\times S_{n-l}}^{S_n} \Big(H^{\mathfrak{b}}(G^n)\boxtimes\mathbb{Q}\Big)$$
$$\hspace{29 mm}= \text{Ind}_{S_l\times S_{n-l}}^{S_n} 
\Big(\text{Ind}_{H\times S_{n-l}}^{S_l\times S_{n-l}} \big( H^{\mathfrak{b}}(G^n)\boxtimes\mathbb{Q}\big)\Big)$$
$$\hspace{21 mm}=\text{Ind}_{S_l\times S_{n-l}}^{S_n} \Big(\big(\text{Ind}_{H}^{S_l} H^{\mathfrak{b}}(G^n)\big)\boxtimes\mathbb{Q}\Big)$$
$$=\text{Ind}_{S_l\times S_{n-l}}^{S_n} V_{\alpha}\boxtimes\mathbb{Q}\text{ .} $$

Moreover, we notice that the forgetful map $\phi_n$ restricted to the summand $H^{\alpha}(G^n)$ corresponds to the inclusion $$\text{Ind}_{S_l\times S_{n-l}}^{S_n} V_{\alpha}\boxtimes\mathbb{Q}\hookrightarrow \text{Ind}_{S_l\times S_{n+1-l}}^{S_{n+1}} V_{\alpha}\boxtimes\mathbb{Q}.$$ Therefore, by  Lemma \ref{IND}, the consistent sequence $\{H^{\alpha}(G^n)\}$ is monotone and uniformly representation stable with stable range $n\geq 2l$, where $l$ is the length of $\alpha$ and $l\leq q$. The result for $\{H^q(G^n;\mathbb{Q}),\phi_n\}$ then follows from Proposition \ref{SUB_QUO}.
\end{proof}

We illustrate the notation in the previous proof with the concrete case of $G=\mathbb{Z}$.
 
By the K\"{u}nneth formula we have $$H^q(\mathbb{Z}^n;\mathbb{Q})\approx\bigoplus_{\sum a_i=q} H^{a_1}(\mathbb{Z};\mathbb{Q})\otimes\cdots\otimes H^{a_n}(\mathbb{Z};\mathbb{Q}).$$

Following our previous notation we take the $n$-tuple $\mathfrak{b}=(1,\ldots,1,0,\ldots,0)$ with $\vline \text{supp}(\mathfrak{b})\vline= q$ and $\alpha:=\overline{\mathfrak{b}}$. Since $H^q(\mathbb{Z};\mathbb{Q})=\mathbb{Q}$ for $q=0,1$ and zero otherwise,  we have that
$$H^q(\mathbb{Z}^n;\mathbb{Q})= \bigoplus_{\overline{\mathfrak{a}}=\alpha} H^{\mathfrak{a}}(\mathbb{Z}^n)=\text{Ind}_{\text{Stab}(\mathfrak{b})}^{S_n} H^{\mathfrak{b}}(\mathbb{Z}^n).$$

\noindent Notice that $\text{Stab}(\mathfrak{b})=S_q\times S_{n-q}$. The corresponding  $(S_q\times S_{n-q})$-representation is 

$$ H^{\mathfrak{b}}(\mathbb{Z}^n)=H^{1}(\mathbb{Z};\mathbb{Q})\otimes\cdots\otimes H^{1}(\mathbb{Z};\mathbb{Q})\otimes\cdots\otimes H^0(\mathbb{Z};\mathbb{Q})\approx V_{\alpha}\boxtimes\mathbb{Q} $$

\noindent where   $V_{\alpha}:= H^{1}(\mathbb{Z};\mathbb{Q})\otimes\cdots\otimes H^{1}(\mathbb{Z};\mathbb{Q})\approx H^{\mathfrak{b}}(\mathbb{Z}^n)$ is regarded as an $S_q$-representation.
Then, as an induced representation,

$$H^q(\mathbb{Z}^n;\mathbb{Q})=\text{Ind}_{S_q\times S_{n-q}}^{S_n}V_{\alpha}\boxtimes\mathbb{Q}.$$
 
Moreover, if $\mathbb{Q}^n$  denotes the permutation $S_n$-representation, then
$$\text{Ind}_{S_q\times S_{n-q}}^{S_n}V_{\alpha}\boxtimes\mathbb{Q}={\bigwedge}^q (\mathbb{Q}^n)={\bigwedge}^q \big(V(0)_n\oplus V(1)_n\big)$$ 
$$\hspace{28mm} =\Big({\bigwedge}^q V(1)_n\Big)\oplus\Big({\bigwedge}^{q-1} V(1)_n\Big)$$
$$\hspace{25mm}= V(\underbrace{1,\ldots,1}_{q})_n\oplus V(\underbrace{1,\ldots,1}_{q-1})_n $$

\noindent Hence, we see explicitly how uniform multiplicity stability holds for this particular case.

\section{Classifying spaces for diffeomorphism groups}\label{LAST}

In this last section we see how the same ideas also imply  representation stability for the cohomology of classifying spaces for diffeomorphism groups.

Let $M$ be a connected and compact smooth manifold of dimension $d\geq 3$. We denote by $\mathcal{E}(M,\mathbb{R}^\infty)$ the space of smooth embeddings $M\rightarrow \mathbb{R}^{\infty}$. It is a contractible space and $\Diff(M\text{ rel }\partial M)$ acts freely by pre-composition. The quotient space $\mathcal{E}(M,\mathbb{R}^\infty)/ \Diff(M\text{ rel }\partial M)$ is %the space of smooth submanifolds of $\mathbb{R}^{\infty}$ diffeomorphic to $M$ and it is 
a classifying space $\BDiff(M\text{ rel }\partial M))$ for the group $\Diff(M\text{ rel }\partial M))$. Similarly we can consider the action of the subgroup $\PDiff^n(M)$ of  $\Diff(M\text{ rel }\partial M)$ (defined in the Introduction) on  $\mathcal{E}(M,\mathbb{R}^\infty)$. The quotient space is a classifying space $\BPDiff^n(M)$ for $\PDiff^n(M)$ and we have a fiber bundle
\begin{equation}\label{FIB}
\BPDiff^n(M)\rightarrow \BDiff(M\text{ rel }\partial M)
\end{equation}
\noindent where the fiber is given by $\Diff(M\text{ rel }\partial M)/\PDiff^n(M) \approx C_n(M)$, the configuration space of $n$ ordered points in $M$.

On the other hand we can consider the forgetful homomorphism $\PDiff^{n+1}(M)\rightarrow \PDiff^n(M)$, which induces a corresponding map between classifying spaces $$f_n: \BPDiff^{n+1}(M)\rightarrow \BPDiff^n(M).$$

There is a Leray-Serre spectral sequence associated to the fiber bundle (\ref{FIB}) that converges to the cohomology $H^*(\BPDiff^n(M);\mathbb{Q})$ with $E_2$-page given by
\begin{equation}\label{LERAY}
 E_2^{p,q}(n)=H^p\big(\BDiff(M\text{ rel }\partial M);H^q(C_n(M);\mathbb{Q})\big).
\end{equation}

\noindent Here, we regard (\ref{LERAY}) as the $p$th cohomology group of $\BDiff(M\text{ rel }\partial M)$ with local coefficients in the $G$-module $H^q(C_n(M);\mathbb{Q})$, where $G=\pi_1(\BDiff(M\text{ rel }\partial M))$ (see \cite[Section 3.H]{HATCHER}). Notice that the actions of $S_n$ and $G$ on $H^q(C_n(M);\mathbb{Q})$ commute. Therefore  $\{H^q(C_n(M);\mathbb{Q})\}_{n=1}^{\infty}$ is a consistent sequence compatible with $G$-actions. Moreover, by  Theorem \ref{CHURCH_RS}, it is monotone and uniformly representation stable, with stable range $n\geq 2q$. Monotonicity and uniform representation stability for the terms in the $E_2$-page will be a consequence of the following result, which is essentially Theorem \ref{GSTAB} from before.

\begin{theo}[Representation stability with changing coefficients 2]\label{GSTAB2} Let $G$ be the fundamental group of a connected CW complex $X$ with finitely many cells in each dimension. Consider a consistent sequence $\{V_n,\phi_n\}_{n=1}^{\infty}$ of finite dimensional rational representations of $S_n$ compatible with $G$-actions. If the sequence $\{V_n,\phi_n\}_{n=1}^{\infty}$ is monotone and uniformly representation stable with stable range $n\geq N$, then for any non-negative integer $p$, the sequence of cohomology groups with local coefficients $\{H^p(X;V_n), \phi_n^*\}_{n=1}^{\infty}$ is monotone and uniformly representation stable with the same stable range.
\end{theo}
\begin{proof}
Since $G=\pi_1(X)$, the universal cover $\tilde{X}$ of $X$  has a $G$-equivariant cellular chain complex. Given that $X$ has finitely many cells in each dimension, for each $p$ the group  $C_p(\tilde{X})$ is a free $G$-module of finite rank, where a preferred $G$-basis can be provided by selecting a $p$-cell in $\tilde{X}$ over each $p$-cell in $X$. Hence, the proof of Theorem \ref{GSTAB2} is the same as the one for Theorem \ref{GSTAB}, by replacing the notions of cohomology of groups by cohomology of a space with local coefficients. 
\end{proof}

Hence when $\BDiff(M \text{ rel } \partial M)$ has the homotopy type of a CW-complex with finitely many cells in each dimension, we can apply the inductive argument from Section \ref{proof} on the successive pages of the Leray-Serre spectral sequence from above and obtain the following result.

\begin{lemma}\label{E2DIM>3} For every $i\geq 0$ and every $n\geq 2$, the consistent sequence of rational $S_n$-representations $$\{ E_2^{i-q,q}(n)=H^{i-q}\big(\BDiff(M\text{ rel }\partial M);H^q(C_n(M);\mathbb{Q})\big)\}_{n=1}^{\infty}$$ is monotone and uniformly representation stable with stable range $n\geq 2q$. Furthermore $E_{\infty}^{i-q,q}(n)=E_{i+2}^{i-q,q}(n)$, which is monotone and uniformly representation stable with stable range $$n\geq 2q+2(i+1)(i).$$
\end{lemma}

As a consequence we get Theorem \ref{MAINBDIFF>3} for the cohomology of the classifying space of a group of diffeomorphisms.

If the manifold $M$ is orientable, we can replace $\Diff(M \text{ rel } \partial M)$ by the group of orientation-preserving diffeomorphims $\Diff^+(M \text{ rel } \partial M)$ in the above argument. In particular, Hatcher and McCullough proved in \cite{HATCHER_MC} that if $M$ is an irreducible, compact connected orientable $3$-manifold with nonempty boundary, then  $\BDiff^+(M\text{ rel }\partial M)$ is a finite K$(\pi, 1)$-space for the mapping class group $\Mod(M)$. Therefore, Theorem \ref{MAINBDIFF>3} is true for this type of manifold. Moreover, if $M$ satisfies conditions (i)-(iv) in \cite[Section 3]{HATCHER_MC}, then $\pi_1(M)$ is centerless and we can apply Theorem  \ref{MAINDIM>3} to get uniform representation stability for the cohomology of $\PMod^n(M)$.       
\bibliographystyle{amsalpha}
\bibliography{referPureMCG}\bigskip\bigskip

%\begin{small}
%\noindent Department of Mathematics\\
%University of Chicago\\
%5734 University Ave.\\
%Chicago, IL 60637\\
%E-mail: \texttt{\href{mailto:atir83@math.uchicago.edu}{atir83@math.uchicago.edu}}
%\end{small}

 \end{document}